\documentclass[11pt,a4paper]{amsart}
\usepackage{graphicx,multirow,array,amsmath,amssymb,upgreek,color,mathtools,extarrows,kotex}
\usepackage[colorlinks=true,linkcolor=red, citecolor=blue]{hyperref}

\newtheorem{theorem}{Theorem}
\newtheorem{proposition}[theorem]{Proposition}
\newtheorem{lemma}[theorem]{Lemma}

\usepackage{multirow}
\usepackage[table,xcdraw]{xcolor}

\begin{document}

\title{The minor minimal intrinsically chiral graphs}
\author[H. Choi]{Howon Choi}
\address{Department of Mathematics, Korea University, Seoul 02841, Korea}
\email{howon@korea.ac.kr}
\author[H. Kim]{Hyoungjun Kim}
\address{Institute of Mathematical Sciences, Ewha Womans University, Seoul 03760, Korea}
\email{kimhjun@korea.ac.kr}
\author[S. No]{Sungjong No}
\address{Department of Mathematics, Kyonggi University, Suwon 16227, Korea}
\email{sungjongno@kgu.ac.kr}

\thanks{2010 Mathematics Subject Classification: 05C10, 57M15}
\thanks{Key words and phrases: spatial graph, intrinsically chiral, minor minimal, symmetry}

\begin{abstract}
Molecular chirality is actively researched in a variety of areas of biology, including biochemistry, physiology, pharmacology, etc., and today many chiral compounds are widely known to exhibit biological properties.
The molecular structure is represented by a graph structure.
Therefore, the study of the mirror symmetry of a graph is important in the natural sciences.
A graph $G$ is said to be intrinsically chiral if no embedding of G is ambient isotopic to its mirror image.

In this paper, we find two minor minimal intrinsically chiral graphs $\Gamma_7$ and $\Gamma_8$.
Furthermore, we classify all intrinsically chiral graphs with at most eleven edges.
\end{abstract}

\maketitle

\section{Introduction}

In this paper, we will take an embedded graph to mean a graph embedded in $S^3$.
Conway and Gordon~\cite{CG} showed that every embedding of the complete graph $K_6$ contains a non-splittable link and every embedding of $K_7$ contains a non-trivial knot.
These properties are called {\it intrinsic\/} properties of the graph.
We say that $K_6$ is {\it intrinsically linked\/} and $K_7$ is {\it intrinsically knotted\/}.
A graph $H$ is a {\it minor\/} of another graph $G$ if it can be obtained by contracting edges in a subgraph of $G$.
If an intrinsically linked graph $G$ has no proper minor that is intrinsically linked, we say $G$ is {\it minor minimal intrinsically linked\/}.
A {\it minor minimal intrinsically knotted} graph is also defined similarly.
Robertson and Seymour's~\cite{RS} Graph Minor Theorem implies that there are only finitely many minor minimal intrinsically knotted graphs.

The study of spatial graph theory is closely related to not only the intrinsic properties of graphs, but also the symmetry of non-rigid molecules.
In chemistry, it is important to predict chemical properties in order to know whether a molecule is different from its mirror image.
A molecule is said to be {\it chemically chiral\/} if it cannot convert itself into its mirror image, otherwise it is said to be {\it chemically achiral}.
In spatial graph theory, molecular structures are interpreted as graphs embedded in $S^3$.
An embedding of a graph $G$ in $S^3$ is {\it topologically achiral\/} if it is ambient isotopic to its mirror image, otherwise it is {\it topologically chiral\/}.
A graph $G$ is {\it intrinsically chiral\/} if every embedding of $G$ in $S^3$ is topologically chiral, otherwise it is {\it achirally embeddable\/}.

\begin{figure}[h!]
\center
\includegraphics[scale=1]{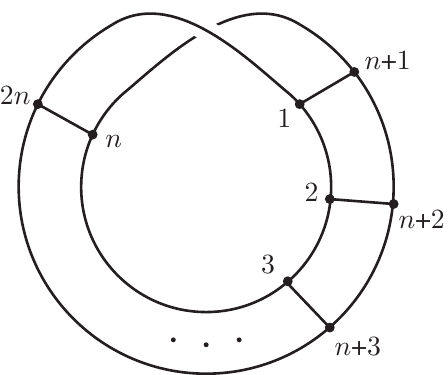}
\caption{A standard embedding of $M_n$}
\label{fig:mob}
\end{figure}

A {\it M{\"o}bius ladder\/}, denoted by $M_n$, is the graph consisting of a simple closed curve $K$ with $2n$ vertices, together with $n$ additional edges $\alpha_1 , \dots , \alpha_n$ such that if the vertices on $K$ are consecutively labeled $1,2,3,\dots,2n$ then the vertices of each edge $\alpha_i$ are $i$ and $i+n$.
We say $K$ is the {\it loop\/} of $M_n$ and $\alpha_1 , \dots , \alpha_n$ are the {\it rungs\/} of $M_n$.
We also say that $M_n$ is {\it standardly embedded\/} in $S^3$, if $M_n$ is embedded as in Figure~\ref{fig:mob} or its mirror image.
Simon~\cite{S} showed that for every standard embedding of $M_n$, there is no orientation reversing diffeomorphism $h$ of $S^3$ with $h(M_n)=M_n$ and $h(K)=K$ when $n \ge 4$.
The condition $h(K)=K$ is required for the chemical motivation, because a loop $K$ represents a molecular chain and it is different from the rungs which represent molecular bonds.
Flapan~\cite{Fl1} showed that every embedding of $M_n$, for $n \ge 3$ odd, there is no orientation reversing diffeomorphism $h$ of $S^3$ with $h(M_n)=M_n$ and $h(K)=K$.
By using this result, Flapan and Weaver~\cite{FW} showed that the complete graphs $K_{4n+3}$ with $n \geq 1$ are intrinsically chiral, and all other complete graphs $K_n$ are achirally embeddable.
Flapan and Fletcher~\cite{FF} classified which complete multipartite graphs admit topologically achiral embeddings.

\begin{figure}[h!]
\center
\includegraphics[scale=1]{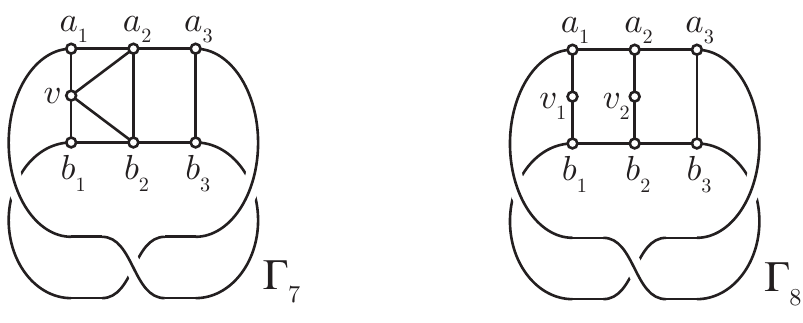}
\caption{Two minor minimal intrinsically chiral graphs $\Gamma_7$ and $\Gamma_8$}
\label{fig:mmic}
\end{figure}

In spatial graph theory, classifying all minor minimal intrinsically linked or knotted graphs is an essential problem, and is the focus of active research~\cite{BM, Fo1, Fo2, GMN, KMO, LKLO, RST}.
Similar with minor minimal intrinsically knotted or linked, we can define minor minimal intrinsically chiral. 
If a graph $G$ is intrinsically chiral and has no proper minor that is intrinsically chiral, we say $G$ is {\it minor minimal intrinsically chiral\/}.

A graph property $\mathcal{P}$ is {\it minor-closed} if every minor of a graph in $\mathcal{P}$ is also in $\mathcal{P}$.
It is important to classify the complete minor minimal family of graphs which does not have the minor-closed property, since this family is useful to determine if other graphs have the property or not.
By Wagner's Theorem~\cite{W}, planar is a minor-closed property and the complete minor minimal family of non-planar graphs consists of $K_5$ and $K_{3,3}$.
This implies that a graph $G$ has $K_5$ or $K_{3,3}$ as minor if and only if $G$ is non-planar.
Note that linklessly embeddable and knotlessly embeddable properties are minor-closed.
By Flapan and Weaver~\cite{FW}, achirally embeddable is not minor-closed, since $K_8$ is achirally embeddable but $K_7$ is not.

Although achirally embeddable is not a minor-closed property it is still
useful to find minor minimal intrinsically chiral graphs as they give a warrant that certain graphs are achirally embeddable: If a graph has no minor
that is minor minimal intrinsically chiral, then that graph is achirally embeddable. On the other hand, every intirinsically chiral graph has a minor
that is minor minimal intrinsically chiral. 
Our overall goal is to classify the minor minimal intrinsically chiral graphs.
Let $\Gamma_7$ be a graph which consists of seven vertices and twelve edges, and $\Gamma_8$ be a graph which consists of eight vertices and eleven edges as drawn in Figure~\ref{fig:mmic}.
Our main theorem is as follows.

\begin{theorem}\label{thm:mmic}
The graphs $\Gamma_7$ and $\Gamma_8$ are minor minimal intrinsically chiral.
\end{theorem}

To show the main theorem, we first find the minimum number of vertices and edges for a simple graph to be intrinsically chiral.

\begin{theorem}\label{thm:main}
If a simple graph $G$ is intrinsically chiral, then $G$ has at least seven vertices and at least eleven edges.
\end{theorem}

In Section~\ref{sec:mincon}, we show that the two graphs $\Gamma_7$ and $\Gamma_8$ are intrinsically chiral, and they satisfy the minimum conditions of vertices and edges for intrinsic chirality, respectively.
Since $\Gamma_8$ has two degree 2 vertices, we also find such conditions for a simple graph which consists of vertices with degree 3 or more.

\begin{theorem}\label{thm:main2}
If a simple graph $G$ is intrinsically chiral which consists of vertices with degree 3 or more, then $G$ has at least seven vertices and at least twelve edges.
\end{theorem}

In the next section we introduce some terminology and outline the strategy of our proof.

\section{Terminology}

A simple graph is a graph containing no loops and multiple edges.
Let $G=(V,E)$ denote a simple graph, where $V$ and $E$ are the sets of vertices and edges, respectively.
The degree of a vertex $v$ in $G$ is the number of edges incident to $v$, denoted by $\deg(v)$.
The minimum degree of $G$, denoted by $\delta(G)$, is the minimum degree among all vertices of $G$.
Let $|G|$ and $||G||$ denote the number of vertices and edges of $G$, respectively.
If $G$ is a disconnected intrinsically chiral graph, then it must have a component which is intrinsically chiral.
This means that it has an intrinsically chiral graph as minor, and so $G$ is not MMIC.
Therefore, this paper deals only with simple connected graphs because the purpose of this paper is to find the minimum conditions of vertices and edges for a simple graph to be intrinsically chiral.

Suppose that an embedding of the graph $G$ is symmetrical on the left and right with respect to a plane $\mathcal{M}$.
Let $h$ be a reflection through $\mathcal{M}$.
Then $h$ is an orientation reversing homeomorphism which sends a such embedding of $G$ to itself.
Hence this embedding is topologically achiral.
This means that $G$ is not intrinsically chiral.
We call $\mathcal{M}$ a mirror plane.
Note that every planar graph has a topologically achiral embedding, since it can lie on the mirror plane.

\begin{figure}[h!]
\center
\includegraphics[scale=1.25]{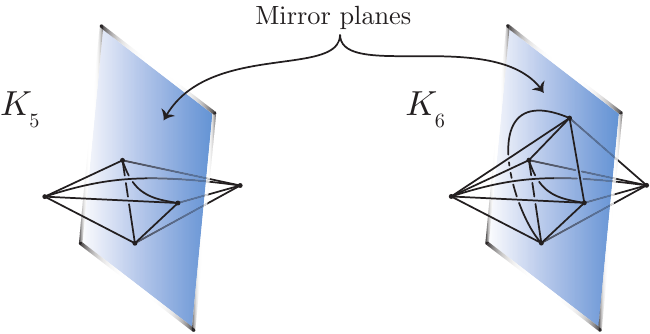}
\caption{Mirror symmetry embeddings of $K_5$ and $K_6$}
\label{fig:mirror}
\end{figure}

Flapan and Weaver~\cite{FW} showed that each of the complete graphs $K_5$ and $K_6$ has a topologically achiral embedding.
Both of them have embeddings which are symmetrical about a mirror plane $\mathcal{M}$ as drawn in Figure~\ref{fig:mirror}.
By using this property, we can show that $K_5$ and $K_6$ are not intrinsically chiral.
This property will frequently be used in the rest of the paper.

It is also important to show that the given graph is intrinsically chiral.
We use following proposition, which is shown by Flapan~\cite{Fl1}, to determine the intrinsic chirality of the given graph.

\begin{proposition}\label{prop:fla}
Let $M_n$ be a M{\"o}bius ladder which is embedded in $S^3$ with loop $K$, where $n \geq 3$ is an odd number.
Then there is no diffeomorphism $h:S^3 \rightarrow S^3$ which is orientation reversing with $h(M_n)=M_n$ and $h(K)=K$.
\end{proposition}

In Section~\ref{sec:mincon}, we consider a graph automorphism in order to apply this propositions to show that $\Gamma_7$ and $\Gamma_8$ are intrinsically chiral.
A graph {\it automorphism} of $G$ is a permutation $\phi$ on the set of vertices $V$ that satisfies the property that $\{ u_i, u_j \} \in E$ if and only if $\{ \phi (u_i), \phi (u_j) \} \in E$.

\section{Proof of Theorem~\ref{thm:main}}

Let $G$ be a simple graph.
It sufficient to show that if $|G| \le 6$ or $||G|| \le 10$ then $G$ is achirally embeddable.
First, we consider the case that $|G| \le 6$.
By using the complement graph of $G$, we classify all non-planar graphs.
A graph $H$ is said to be the {\it complement} of a graph $G$ if $H$ has the same set of vertices with $G$ and has the complement set of edges of $G$.
By following lemma, every graph in this case, is achirally embeddable.

\begin{lemma}\label{lem:6v}
Every graph with at most 6 vertices is achirally embeddable.
\end{lemma}

\begin{proof}
If G has at most 5 vertices, then $G$ is either a planar graph or the complete graph $K_5$.
Since $K_5$ and any planar graph are achirally embeddable, $G$ is not intrinsically chiral.

\begin{figure}[h!]
\center
\includegraphics[scale=1]{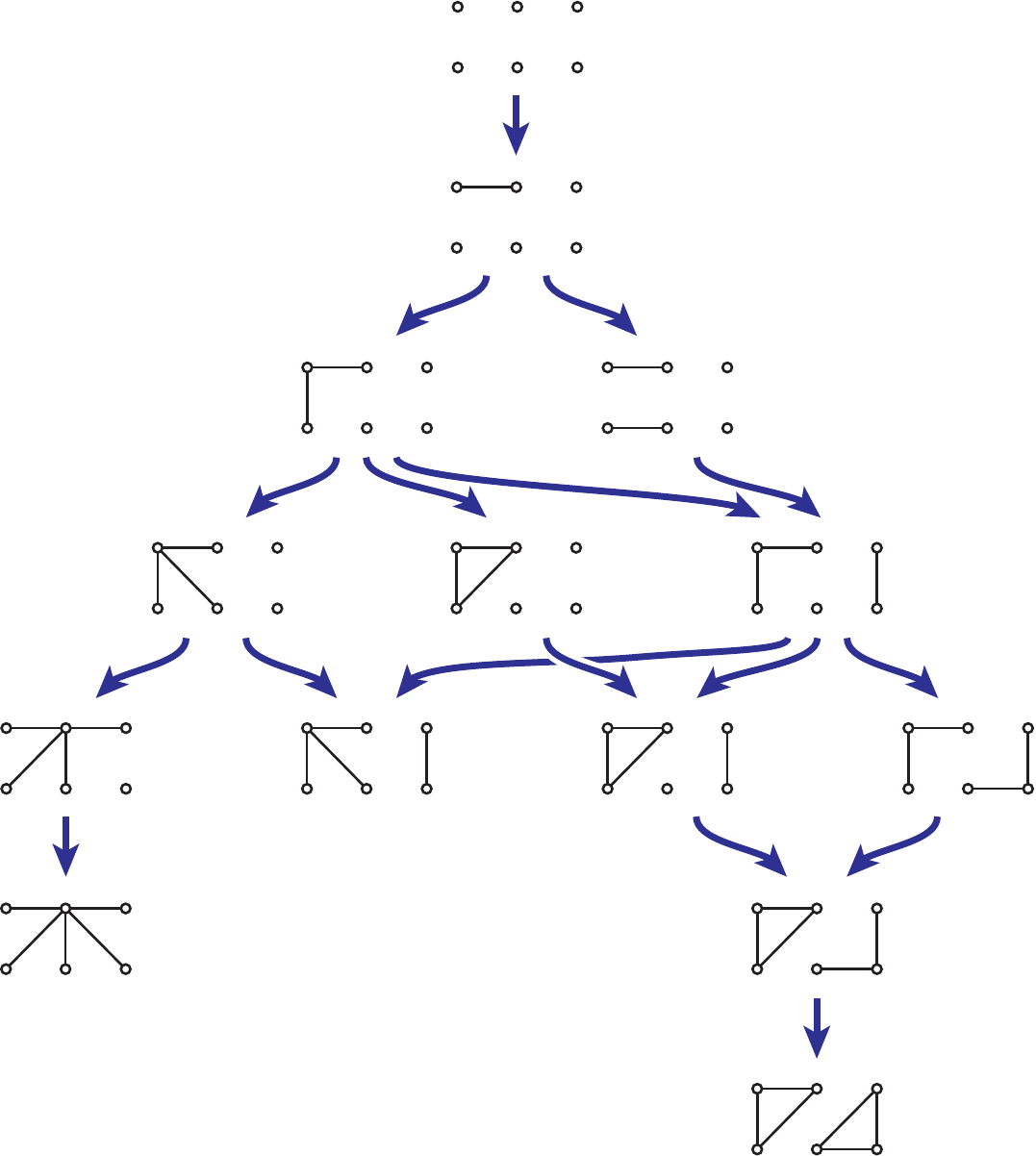}
\caption{The complement graphs of non-planar subgraphs of $K_6$}
\label{fig:13c}
\end{figure}

It remains to consider that $G$ has 6 vertices.
Now we classify every non-planar graph with 6 vertices by deleting edges from $K_6$.
There are 14 non-planar graphs which are subgraphs of $K_6$.
Figure~\ref{fig:13c} represents the complement graph of non-planar subgraphs of $K_6$.
In this figure, we find all such non-planar graphs by adding a removed edge from $K_6$ repeatedly.
Note that the graph obtained from a planar graph is also planar.
So if we obtain a planar graph in this process then we do not consider it.
Since each of them has a topologically achiral embedding as drawn in Figure~\ref{fig:14g}, $G$ is not intrinsically chiral.
\end{proof}

\begin{figure}[h!]
\center
\includegraphics[scale=0.8]{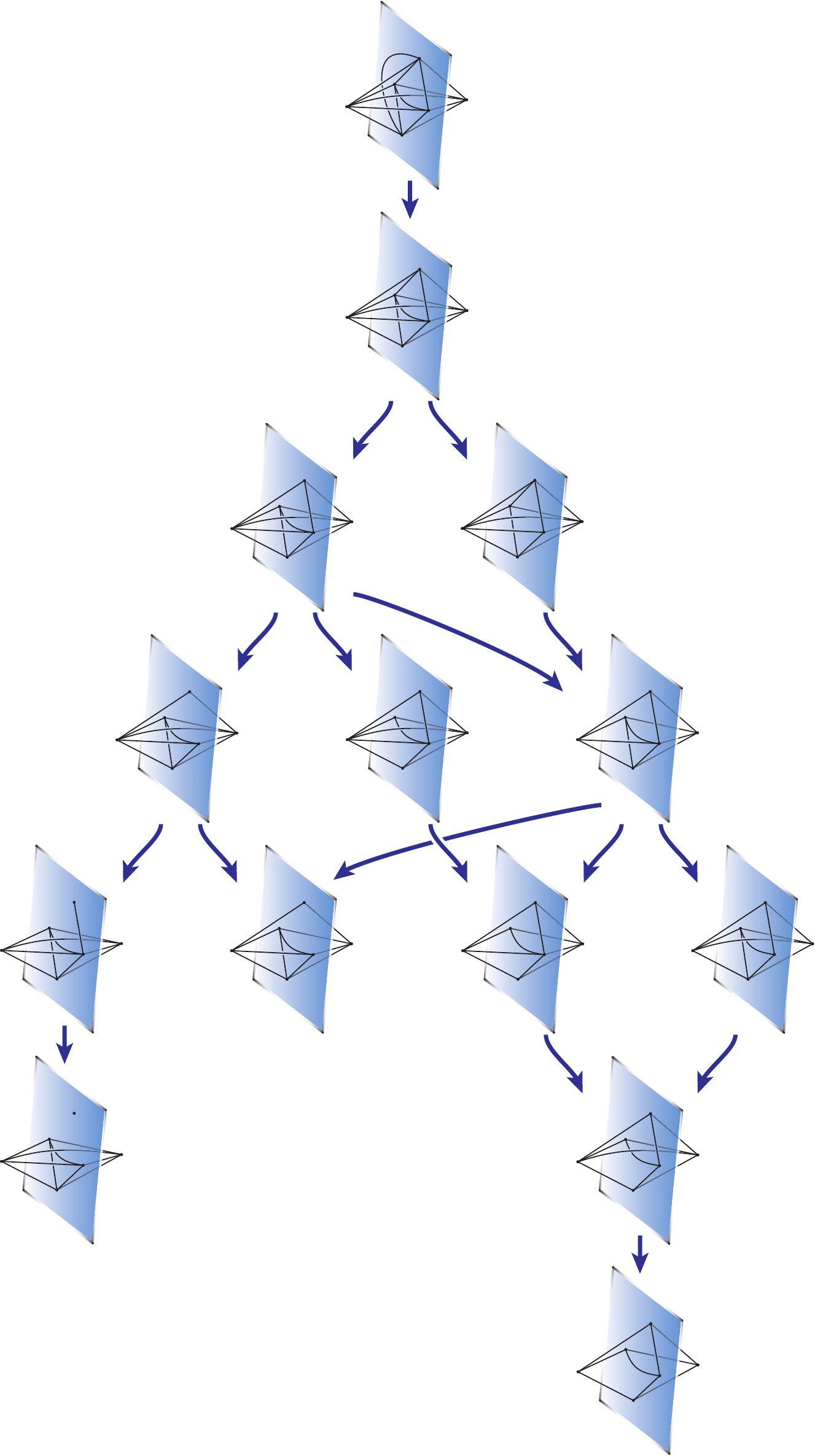}
\caption{Topologically achiral embeddings of non-planar subgraphs of $K_6$}
\label{fig:14g}
\end{figure}

We can determine $G$ is planar or not by the relation between $||G||$ and $|G|$. 
For example, if a connected graph $G$ satisfies $||G|| - |G| = -1$, then $G$ is tree, and so it is planar.
The following lemma helps us to prove Theorem~\ref{thm:main}.

\begin{lemma}\label{lem:pla}
Let $G$ be a connected graph.
If $||G||-|G| \le 2$, then $G$ is planar.
\end{lemma}

\begin{proof}
Suppose for the contradiction that $G$ is non-planar.
Then we can obtain $K_{3,3}$ or $K_5$ from $G$ by a sequence of edge deletion or edge contraction.
Note that $||G||-|G|$ is decreased after using an edge deletion, and is not changed after using an edge contraction.
Thus $||G||-|G| \ge 3$, because $||K_{3,3}||-|K_{3,3}| =3$ and $||K_5||-|K_5|=5$.
This contradicts the assumption.
\end{proof}

Now, we consider the case that $||G|| \le 10$.
By following lemma, every graph in this case is achirally embeddable.

\begin{lemma}\label{lem:10}
Every graph with at most 10 edges is achirally embeddable.
\end{lemma}

\begin{proof}
Suppose that $G$ has at most $10$ edges and contains $K_5$ or $K_{3,3}$ as a minor.
By Lemma~\ref{lem:6v} and \ref{lem:pla}, we may assume that $G$ is a non-planar graph with 7 vertices and 10 edges.
Since $||G||-|G|=3$, $K_{3,3}$ may be obtained by a single edge contraction of $G$.
This implies that $G$ is either $E$ or $E'$ which are achirally embeddable as drawn in Figure~\ref{fig:7v10e}.
Therefore, $G$ is achirally embeddable.
\end{proof}

\begin{figure}[h!]
\center
\includegraphics[scale=1.25]{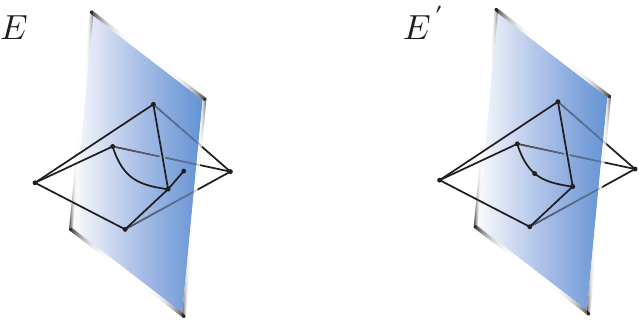}
\caption{Topologically achiral embeddings of $E$ and $E'$}
\label{fig:7v10e}
\end{figure}

By the above lemmas, for $G$ to be an intrinsically chiral graph, it must have at least 7 vertices and 11 edges.

\section{The minimum condition of $|G|$ and $||G||$}\label{sec:mincon}

This section presents two intrinsically chiral graphs $\Gamma_7$ and $\Gamma_8$ that are related to the minimum number of vertices or edges for intrinsic chirality.
This implies that they are minor minimal intrinsically chiral graphs.
To show these graphs are intrinsically chiral, we use Proposition~\ref{prop:fla} for $M_3$ as drawn in Figure~\ref{fig:m3}.
This proposition implies that no orientation reversing diffeomorphism sends the loop of $M_3$ to itself.
By using this result, we prove following propositions.

\begin{figure}[h!]
\center
\includegraphics[scale=1.1]{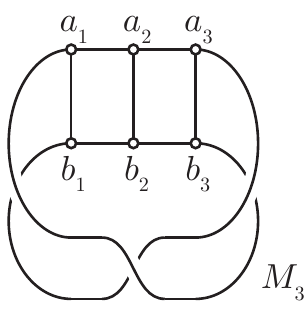}
\caption{The m{\"o}bius ladder $M_3$}
\label{fig:m3}
\end{figure}

\begin{proposition}\label{prop:a7}
The graph $\Gamma_7$ is an intrinsically chiral graph.
\end{proposition}

\begin{proof}
Let $h:S^3 \rightarrow S^3$ be a homeomorphism such that $h(\Gamma_7) = \Gamma_7$.
The homeomorphism $h$ induces an automorphism on the vertices of $\Gamma_7$.
$\Gamma_7$ consists of three degree 4 vertices $v$, $a_2$ and $b_2$, and four degree 3 vertices $a_1$, $b_1$, $a_3$ and $b_3$.
More precisely, $a_1$ and $b_1$ are adjacent to two degree 4 vertices and one degree 3 vertex, and $a_3$ and $b_3$ are adjacent to one degree 4 vertex and two degree 3 vertices.
This means that $h$ does not send $\{ a_1, b_1 \}$ into $\{ a_3, b_3 \}$.
Furthermore $v$, $a_2$ and $b_2$ are adjacent to two degree 4 vertices and two degree 3 vertices.
In detail, $v$ is adjacent to two degree 3 vertices $a_1$ and $b_1$, and $a_2$ is adjacent to two degree 3 vertices $a_1$ and $a_3$.
This implies that $h$ does not send $v$ into $a_2$, since $h(a_1) \neq a_3$ and $h(b_1) \neq a_3$.
Similarly, $h$ does not send $v$ into $b_2$.
Thus we have that 
\begin{eqnarray*}
  \{a_1,~b_1\} \xlongrightarrow{h} \{a_1,~b_1\}\\
  \{a_2,~b_2\} \xlongrightarrow{h} \{a_2,~b_2\}\\
  \{a_3,~b_3\} \xlongrightarrow{h} \{a_3,~b_3\}\\
  \{ v \} \xlongrightarrow{h} \{ v \} \qquad 
\end{eqnarray*}
Let $K=a_1a_2a_3b_1b_2b_3a_1$ be a cycle of $\Gamma_7$.
Then we have
\begin{eqnarray*}
h(a_1a_2a_3b_1b_2b_3a_1) =
\begin{cases}
 a_1a_2a_3b_1b_2b_3a_1 & \text{if~}h(a_1)=a_1,\\
 b_1b_2b_3a_1a_2a_3b_1 & \text{if~}h(a_1)=b_1.
\end{cases}
\end{eqnarray*}
Thus $h(K)=K$ in $\Gamma_7$.

$M_3$ is obtained from $\Gamma_7$ by deleting two edges $e_{va_2}$ and $e_{vb_2}$ and removing a vertex $v$ by edge contraction.
Hence the homeomorphism $h$ satisfies $h(M_3)=M_3$ and $h(K)=K$.
By Proposition~\ref{prop:fla}, $h$ can not be an orientation reversing homeomorphism.
Therefore, $\Gamma_7$ is intrinsically chiral.
\end{proof}

Together with Proposition~\ref{prop:a7}, Theorem~\ref{thm:main} establishes the minimum condition of vertices and edges for any simple graph to be intrinsically chiral.
Note that $\Gamma_7$ has twelve edges.
This is the minimum condition of edges for intrinsic chirality of graphs which consists of vertices with degree 3 or more.
Theorem~\ref{thm:main2} explicitly shows that the minimum condition of edges for such graph.

\begin{figure}[h!]
\center
\includegraphics[scale=1.1]{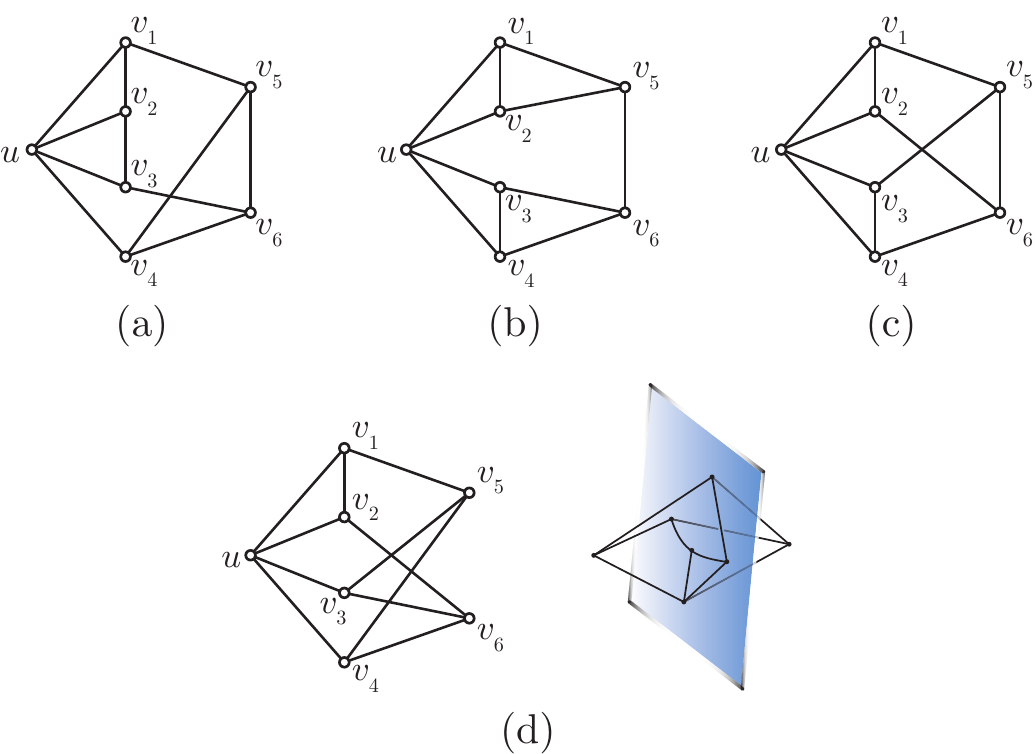}
\caption{4 cases of $G$ when $(|G|,||G||)=(7,11)$}
\label{fig:7v11e}
\end{figure}

\begin{proof}[Proof of Theorem~\ref{thm:main2}]
By Theorem~\ref{thm:main}, we only need to consider the cases that $G$ has 11 edges and 7 or more vertices.
If $G$ has 8 or more vertices then $G$ has at least 12 edges, since $\delta(G) \ge 3$.
So we may assume that $G$ has 7 vertices, and so $G$ consists of a degree 4 vertex and six degree 3 vertices.
Let $u$ be a degree 4 vertex, and $v_1, \dots, v_6$ be degree 3 vertices such that $v_1$, $v_2$, $v_3$, and $v_4$ are adjacent to $u$.
If $v_5$ is adjacent to $v_6$, then there are two edges between $v_1,\dots,v_4$.
Assume that $v_2$ is adjacent to $v_1$ and $v_3$.
Then there is the unique graph as drawn in Figure~\ref{fig:7v11e} (a), and it is planar.
Assume that $v_1$ is adjacent to $v_2$, and $v_3$ is adjacent to $v_4$.
Then there are exactly two graphs as drawn in Figure~\ref{fig:7v11e} (b) and (c), and both graphs are planar.
If $v_5$ is not adjacent to $v_6$, then there is the unique graph which is non-planar as drawn in Figure~\ref{fig:7v11e} (d), and it is achirally embeddable.
Therefore if $G$ is intrinsically chiral then it has at least twelve edges.
\end{proof}

The graph $\Gamma_7$ satisfies the minimum condition of vertices for any simple graph.
The graph $\Gamma_8$ satisfies the minimum condition of edges for any simple graph, and it is the unique intrinsically chiral graph among all simple graphs with 11 edges.

\begin{proposition}\label{prop:a8}
The graph $\Gamma_8$ is an intrinsically chiral graph.
Furthermore, any other simple connected graph with at most 11 edges is achirally embeddable.
\end{proposition}

\begin{proof}
First, we show that $\Gamma_8$ is an intrinsically chiral graph.
Suppose that there is a homeomorphism $h:S^3 \rightarrow S^3$ such that $h(\Gamma_8) = \Gamma_8$.
Similar with the proof of Proposition~\ref{prop:a7}, we have the following:

\underline{Case 1}. $a_1$ $\xlongrightarrow{h}$ $a_1$ or $b_1$
\begin{eqnarray*}
  \{a_1,~b_1\} \xlongrightarrow{h} \{a_1,~b_1\}\\
  \{a_2,~b_2\} \xlongrightarrow{h} \{a_2,~b_2\}\\
  \{a_3,~b_3\} \xlongrightarrow{h} \{a_3,~b_3\}\\
  \{ v_1 \} \xlongrightarrow{h} \{ v_1 \} \qquad\\ 
  \{ v_2 \} \xlongrightarrow{h} \{ v_2 \} \qquad 
\end{eqnarray*}

\underline{Case 2}. $a_1$ $\xlongrightarrow{h}$ $a_2$ or $b_2$
\begin{eqnarray*}
  \{a_1,~b_1\} \xlongrightarrow{h} \{a_2,~b_2\}\\
  \{a_2,~b_2\} \xlongrightarrow{h} \{a_1,~b_1\}\\
  \{a_3,~b_3\} \xlongrightarrow{h} \{a_3,~b_3\}\\
  \{ v_1 \} \xlongrightarrow{h} \{ v_2 \} \qquad\\ 
  \{ v_2 \} \xlongrightarrow{h} \{ v_1 \} \qquad 
\end{eqnarray*}

Let $K=a_1a_2a_3b_1b_2b_3a_1$ be a cycle of $\Gamma_8$.
Then we have
\begin{eqnarray*}
h(a_1a_2a_3b_1b_2b_3a_1) =
\begin{cases}
 a_1a_2a_3b_1b_2b_3a_1 & \text{if~}h(a_1)=a_1,\\
 b_1b_2b_3a_1a_2a_3b_1 & \text{if~}h(a_1)=b_1,\\
 a_2a_1b_3b_2b_1a_3a_2 & \text{if~}h(a_1)=a_2,\\
 b_2b_1a_3a_2a_1b_3b_2 & \text{if~}h(a_1)=b_2.
\end{cases}
\end{eqnarray*}
Thus $h(K)=K$ in $\Gamma_8$.
By Proposition~\ref{prop:fla}, $h$ can not be an orientation reversing homeomorphism.
Therefore $\Gamma_8$ is intrinsically chiral.

\begin{figure}[h!]
\center
\includegraphics[scale=1.25]{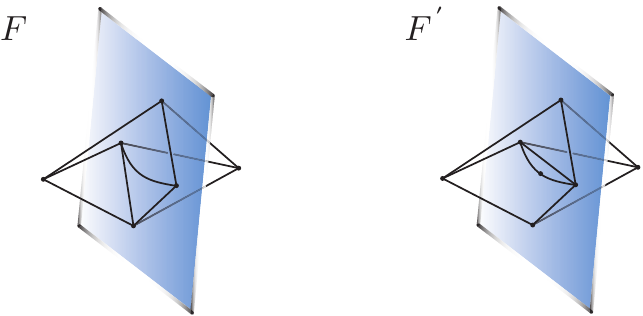}
\caption{Topologically achiral embeddings of $F$ and $F'$}
\label{fig:71112}
\end{figure}

Now suppose that $G$ is an intrinsically chiral graph with at most 11 edges, and $w$ is a vertex with minimal degree of $G$.
By lemma~\ref{lem:6v}, \ref{lem:pla} and \ref{lem:10}, $G$ has 7 or 8 vertices and 11 edges.
By Theorem~\ref{thm:main2}, if $G$ has at most 11 edges and consists of vertices with degree 3 or more, then it can not be intrinsically chiral.
So $w$ must be a degree 1 or 2 vertex.

We first consider the case that $G$ has 7 vertices.
Suppose that $w$ is a degree 1 vertex.
Then $G \setminus w$ consists of six vertices and ten edges.
Since $G$ is non-planar, $G \setminus w$ is the graph $F$ as drawn in Figure~\ref{fig:71112}.
In this case, we can put $w$ and the corresponding edges on the mirror plane of an embedding of $F$.
Hence $w$ is a degree 2 vertex.
Since $G$ is non-planar, $G$ is either the graph obtained from $F$ by a subdivision of an edge, say $e_1$, or the graph $F'$.
In former case, we can obtain an embedding of $F$ such that $e_1$ is on the mirror plane.
From this embedding, we can obtain $G$ by a subdivision of $e_1$, so it is achirally embeddable.
In latter case, $F'$ is achirally embeddable as drawn in Figure~\ref{fig:71112}, so $G$ can not be intrinsically chiral.

Now we consider the case that $G$ has 8 vertices.
Since $||G||-|G|=3$, we may obtain $K_{3,3}$ from $G$ by using edge contractions twice, by Lemma~\ref{lem:pla}.
This means that we may obtain $E$ or $E'$ as drawn in Figure~\ref{fig:7v10e}, from $G$ by contracting an edge adjacent to $w$.
If $\text{deg}(w)=1$ then $G$ is obtained from $E$ or $E'$, as drawn in Figure~\ref{fig:7v10e}, by adding a vertex $w$ and its incident edge.
We can put $w$ and the corresponding edges on the mirror plane of both embeddings of $E$ and $E'$.
If $\text{deg}(w)=2$ then $G$ is obtained from $K_{3,3}$ by twice of subdivisions of edges.
If $G$ is obtained from $K_{3,3}$ by two subdivisions of one edge, say $e_2$, of $K_{3,3}$, then we can obtain an embedding of $K_{3,3}$ such that $e_2$ is on the mirror plane.
From this embedding, we can obtain $G$ by two subdivisions of $e_2$, so it is achirally embeddable.
So $G$ is obtained from $K_{3,3}$ by two subdivisions of two edges, say $e_3$ and $e_4$.
If $e_3$ and $e_4$ has a common vertex then we can obtain an embedding of $K_{3,3}$ such that both edges $e_3$ and $e_4$ are on the mirror plane.
From this embedding, we can obtain $G$ by each subdivision of $e_3$ and $e_4$, so it is achirally embeddable.
Otherwise, $G$ is the graph $\Gamma_8$, and it is the unique intrinsically chiral graph among the graphs with at most 11 edges.
\end{proof}

\begin{proof}[Proof of Theorem~\ref{thm:mmic}]
First consider the minor graphs of $\Gamma_7$.
Since edge contractions and deletions are independent of the order of application,
we can divide these graphs into three types by how many contractions we can get from $G$.
Since the minor graphs obtained from $\Gamma_7$ by using edge contractions more than once are $K_5$ or planar, they are achirally embeddable.
Since the minor graphs obtained from $\Gamma_7$ by only using edge deletions are simple, they are also achirally embeddable by Proposition~\ref{prop:a8}.
So it remains to consider the non-simple minor graphs which are obtained from $\Gamma_7$ by using an edge contraction once.
By the symmetry of edges of $\Gamma_7$, there are four types of edges, $e_{v a_1}$, $e_{v a_2}$, $e_{a_1 a_2}$ and $e_{a_2 b_2}$ where $\Gamma_7$ becomes non-simple when using an edge contraction.
The graphs obtained from $\Gamma_7$ by using an edge contraction at $e_{v a_2}$, $e_{a_1 a_2}$ or $e_{a_2 b_2}$ are planar.
The graph obtained from $\Gamma_7$ by using an edge contraction at $e_{v a_1}$ is not planar, but it is achirally embeddable.
Moreover subgraphs of this graph are also achirally embeddable.
So $\Gamma_7$ is a minor minimal intrinsically chiral graph.

Since $\Gamma_8$ contains no triangle, the graph obtained from $\Gamma_8$ by a contraction is still simple.
The graphs obtained from $\Gamma_8$ by using edge contractions more than one are $K_{3,3}$ of planar.
So they are achirally embeddable.
This implies that $\Gamma_8$ does not have any non-simple minor which is intrinsically chiral.
Therefore, $\Gamma_8$ is a minor minimal intrinsically chiral graph by Proposition~\ref{prop:a8}.
\end{proof}

\bibliography{ic} 
\bibliographystyle{abbrv}

\end{document}